\documentclass[12pt]{article}

\usepackage{times}

\usepackage[utf8]{inputenc}
\usepackage{commath}
\usepackage{amsthm, amscd}
\usepackage{amsmath}
\usepackage{amssymb}
\usepackage{cite}
\usepackage{setspace}

\usepackage{bigints}
\usepackage{comment}
\usepackage{mathtools}
\usepackage{mathrsfs}
\usepackage{fancyhdr}

\allowdisplaybreaks[4]

\numberwithin{equation}{section}
\usepackage{etoolbox}
\patchcmd{\thebibliography}
  {\settowidth}
  {\setlength{\itemsep}{0pt plus -10pt}\settowidth}
  {}{}
\apptocmd{\thebibliography}
  {
  }
  {}{}
\makeatletter
\newtheorem*{rep@theorem}{\rep@title}
\newcommand{\newreptheorem}[2]{%
\newenvironment{rep#1}[1]{%
 \def\rep@title{#2 \ref{##1}}%
 \begin{rep@theorem}}%
 {\end{rep@theorem}}}
\makeatother

\theoremstyle{theorem}

\newreptheorem{theorem}{Theorem}
\newtheorem{thm}{Theorem}[section]
\newtheorem*{thm*}{Theorem}
\theoremstyle{definition}
\newtheorem{prop}[thm]{Proposition}
\newtheorem*{prop*}{Proposition}

\newtheorem{lem}[thm]{Lemma}

\newtheorem*{cor*}{Corollary}
\theoremstyle{remark}

\newtheorem{rem}[thm]{Remark}

\title{On Monge-Ampère volumes of direct images.} 

\author
{Siarhei Finski
}

\date{}

\usepackage[%
    left=1in,%
    right=1in,%
    top=1.2in,%
    bottom=1in,%
    paperheight=11in,%
    paperwidth=8.5in%
]{geometry}

\newcommand{\imun} {\sqrt{-1}}

\newcommand{\comp}{\mathbb{C}}
\newcommand{\real}{\mathbb{R}}

\newcommand{\nat}{\mathbb{N}}

\newcommand{\spec}{{\rm{Spec}}}
\newcommand{\dist}{{\rm{dist}}}

\newcommand{\enmr}[1]{\text{End}{(#1)}}

\newcommand{\ccal}{\mathscr{C}}

\newcommand{\dbar}{ \overline{\partial} }

\newcommand{\rk}[1]{{\rm{rk}} ( #1 )}
\newcommand{\tr}[1]{{\rm{Tr}} \big[ #1 \big]}

\newcommand{\scal}[2]{\big< #1, #2 \big>}

\newcommand{\td}{{\rm{Td}}}
\newcommand{\ch}{{\rm{ch}}}

\usepackage{sectsty}

\sectionfont{\fontsize{13}{15}\selectfont}

\makeatletter
\let\origsection\section
\renewcommand\section{\@ifstar{\starsection}{\nostarsection}}

\newcommand\nostarsection[1]
{\sectionprelude\origsection{#1}\sectionpostlude}

\newcommand\starsection[1]
{\sectionprelude\origsection*{#1}\sectionpostlude}

\newcommand\sectionprelude{%
  \vspace{-10px}
}

\newcommand\sectionpostlude{%
  \vspace{-5px}
}
\makeatother

\newenvironment{sciabstract}{}

\begin{document} 
\maketitle

\begin{sciabstract}
  \textbf{Abstract.} This paper is devoted to the study of the asymptotics of Monge-Ampère volumes of direct images associated to high tensor powers of an ample line bundle. We study the leading term of this asymptotics and provide a classification of bundles which saturate the topological bound of Demailly on those volumes.
  In the special case of high symmetric powers of ample vector bundles, this provides a characterization of vector bundles admitting projectively flat Hermitian structures.

\end{sciabstract}

\pagestyle{fancy}
\lhead{}
\chead{On Monge-Ampère volumes of direct images}
\rhead{\thepage}
\cfoot{}


\newcommand{\Addresses}{{
  \bigskip
  \footnotesize
  \noindent \textsc{Siarhei Finski, Institut Fourier - Université Grenoble Alpes, France.}\par\nopagebreak
  \noindent  \textit{E-mail }: \texttt{finski.siarhei@gmail.com}.
}}

\section{Introduction}\label{sect_intro}
	
	In \cite{DemailluHYMGriff}, Demailly proposed an elliptic system of differential equations of Hermitian-Yang-Mills type for the curvature tensor of a vector bundle with an ample determinant. 
	This system of differential equations is designed so that the existence of a solution to it implies the existence of a dual Nakano positive Hermitian metric (see Section \ref{sect_pos_conc} for a precise definition) on the vector bundle.
	\par The first equation of this elliptic system is a determinantal equation for the curvature and it should be regarded as a vector-bundle version of the Monge-Ampère equation. Motivated by this and the theory of volumes of line bundles (which is based on the Monge-Ampère equation), Demailly in \cite[\S 3]{DemailluHYMGriff} introduced the Monge-Ampère volume, ${\rm{MAVol}}(F, h^F)$, of a dual Nakano positive Hermitian vector bundle $(F, h^F)$ on a compact complex manifold $Z$ of dimension $n$ as follows
	\begin{equation}\label{eq_defn_mavol}
		{\rm{MAVol}}(F, h^F) := \int_X \det{}_{TZ \otimes F^*}\big(  (\Theta^{F})^T / 2 \pi \big)^{1/\rk{F}},
	\end{equation}
	where $R^F := (\nabla^F)^2$ is the curvature of the Chern connection $\nabla^F$ of $(F, h^F)$, $\Theta^F := \imun R^F$, and the determinant should be interpreted as follows. 
	Fix $x \in Z$ and some basis $\partial z_1, \ldots, \partial z_n$ of $T^{1, 0}_x Z$. Denote the associated dual basis by $dz_1, \ldots, dz_n$. Then define
	\begin{multline}\label{eq_defn_det}
		\det{}_{TZ \otimes F^*}\big(  (\Theta^{F})^T/2 \pi \big)^{1/\rk{F}}(x)
		:=
		\det{} \big(  \Theta^{F}(\partial z_i, \overline{\partial z}_j)^T/2 \pi \big)^{1/\rk{F}} 
		\cdot
		\\
		\cdot n! \cdot \imun dz_1 \wedge d\overline{z}_1 \wedge \cdots \wedge \imun dz_n \wedge d\overline{z}_n,
	\end{multline}
	where the last determinant is viewed as a determinant of an operator acting on the space $(F^*)^{\oplus n}$.
	Note that this determinant is positive, since $(F, \nabla^F)$ is dual Nakano positive.
	\par 
	Clearly, ${\rm{MAVol}}(F, h^F)$ doesn't depend on $h^F$ if $F$ is a line bundle (and only in this case,  see for example \cite[Remark 3.3c)]{DemailluHYMGriff}), and in this case our normalization gives ${\rm{MAVol}}(F, h^F) = \int_Z c_1(F)^{n}$, where $c_1(F)$ is the first Chern class of $F$.
	\par The main goal of this paper is to study  ${\rm{MAVol}}(E_k, h^{E_k})$ for a sequence of Hermitian vector bundles $(E_k, h^{E_k})$, $k \in \nat$, obtained as direct images of high tensor powers of an ample line bundle. 
	\par More precisely, we consider a holomorphic proper submersion $\pi : Y \to B$ between compact complex manifolds and denote by $X$ its fiber. Let $(L, h^L)$ (resp. $(G, h^G)$) be a Hermitian line (resp. vector) bundle over $Y$. We suppose that $L$ is ample and the curvature of the Chern connection of $(L, h^L)$ is positive. Then by a combination of Kodaira vanishing theorem, Grauert theorem and Riemann-Roch-Hirzebruch theorem, for large $k \in \nat$, the cohomology of the fibers 
	\begin{equation}
		E_k := H^0(X, L^{\otimes k} \otimes G),
	\end{equation}
	form vector bundles over the base of the family, $B$.
	\par 
	We fix a relative volume form $\nu_{TY/B}$ of $\pi$ and endow $E_k$ with the $L^2$-Hermitian metric $h^{E_k}$, defined using the pointwise scalar product $\langle \cdot, \cdot \rangle_{h}$, induced by $h^L$ and $h^G$, as follows
	\begin{equation}\label{eq_l2_prod}
		\scal{\alpha}{\alpha'}_{L^2} := \int_X \scal{\alpha(x)}{\alpha'(x)}_{h} d \nu_{TY/B}(x), \qquad \text{for } \, \, \alpha, \alpha' \in E_k.
	\end{equation}
	\par 
	Berndtsson in \cite{BerndAnnMath}, \cite{BernJDG} (or  Ma-Zhang in \cite{MaZhangSuperconn}, cf. Theorem \ref{thm_ma_zhang}), proved that the Hermitian vector bundles $(E_k, h^{E_k})$ are positive in the sense of Nakano (or dual Nakano) for $k$ big enough. In particular, the Monge-Ampère volumes, ${\rm{MAVol}}(E_k, h^{E_k})$, are well-defined for $k$ big enough.
	\begin{sloppypar}
		By the positivity of $\Theta^L$, the form $\omega :=  c_1(L, h^L) := \frac{\Theta^L}{2 \pi}$ defines the orthogonal splitting 
		\begin{equation}\label{eq_splitting}
			TY = (TY/\pi^*TB) \oplus T^H Y,
		\end{equation}
		with a smooth isomorphism $T^H Y \cong \pi^*TB$.
		For $U \in TB$, we denote by $U^H \in T^H Y$ the horizontal lift of $U$ with respect to (\ref{eq_splitting}). 
		We denote by $\omega_H \in \ccal^{\infty}(Y, \pi^*\Lambda^2(TB))$ the section
		\begin{equation}\label{eq_omega_h}
			\omega_H(U,V) :=  \omega(U^H,V^H), \qquad \text{for } \, \, U, V \in TB.
		\end{equation}
	\end{sloppypar}
	\begin{thm}\label{thm_main}
		As $k \to \infty$, the Monge-Ampère volumes satisfy the following asymptotics 
		\begin{equation}\label{eq_main_form}
			{\rm{MAVol}}(E_k, h^{E_k}) 
			\sim
			k^{\dim B}
			\cdot
			\int_B 
			\exp 
			\Big(
				\frac{\int_X \log \big( \omega_H^{\wedge \dim B} / \pi^* \nu_B \big)
				\omega^{\wedge \dim X}}{\int_X c_1(L)^{\dim X}}
			\Big)
			d \nu_B,
		\end{equation}
		where $\nu_B$ is an arbitrary volume form on $B$ (the right-hand side of (\ref{eq_main_form}) doesn't depend on the choice of $\nu_B$). In particular, the asymptotics depends solely on the form $\omega$ and not on $h^G$ or $\nu_{TY/B}$.
	\end{thm}
	\begin{rem}
		One can also define the Monge-Ampère volumes for Nakano positive vector bundles by the same formula as in (\ref{eq_defn_mavol}) but without the transposition on $\Theta^F$. Theorem \ref{thm_main}, as well as other results of this paper, continue to hold for this definition as well.
	\end{rem}
	\par Now, for a given holomorphic vector bundle $F$, Demailly in \cite[Remark 3.3b)]{DemailluHYMGriff} proposed to study metrics $h^F$ achieving the maximum of the Monge-Ampère volumes ${\rm{MAVol}}(F, h^F)$. In the second part of this paper, we study a related problem in the asymptotic sense for $E_k$.
	\par 
	More precisely, let's define for a dual Nakano positive Hermitian vector bundle $(F, h^F)$ on a complex manifold $Z$ of dimension $n$ the rescaled Monge-Ampère volume $\overline{\rm{MAVol}}(F, h^F)$ by
	\begin{equation}\label{eq_defn_ma_renorm}
		\overline{\rm{MAVol}}(F, h^F) := \frac{{\rm{MAVol}}(F, h^F)}{\rk{F}^{-n} \int_Z c_1(F)^{n}}.
	\end{equation}
	This definition is motivated by the result of Demailly \cite[Proposition 3.2]{DemailluHYMGriff}, stating that 
	\begin{equation}\label{eq_resc_ma_bound}
		\overline{\rm{MAVol}}(F, h^F) \leq 1.
	\end{equation}
	\par Now, instead of considering a maximization of ${\rm{MAVol}}(F, h^F)$ as a functional on $h^F$ -- a problem which according to Demailly \cite[Remark 3.3b)]{DemailluHYMGriff} involves solving a fourth order non linear differential system -- we consider an easier problem of classification of $(L, h^L)$, for which the associated $L^2$-metrics $h^{E_k}$ asymptotically saturate the upper bound $1$ on $\overline{\rm{MAVol}}(E_k, h^{E_k})$ from (\ref{eq_resc_ma_bound}).
	\begin{thm}\label{thm_charact_max}
		The following assertions are equivalent:
		 a) As $k \to \infty$, $\overline{\rm{MAVol}}(E_k, h^{E_k}) \to 1$.
		\par b)
		For some $p \in \nat^*$, the Hermitian line bundle $(L^{\otimes p}, (h^{L})^{\otimes p})$ decomposes as 
		\begin{equation}\label{eq_dec_max_char}
			(L^{\otimes p}, (h^{L})^{\otimes p}) = (L_V, h^{L_V}) \otimes  (\pi^* L_H, \pi^* h^{L_H}),
		\end{equation}		 
		where  $L_H$ is an  ample line bundle on $B$, endowed with a Hermitian metric $h^{L_H}$ with positive curvature and a Hermitian line bundle $(L_V, h^{L_V})$ has semi-positive curvature $\Theta^{L_V}$, for which the kernel $\ker(\Theta^{L_V})$ defines a foliation of dimension $\dim B$.
	\end{thm}
	\begin{rem}
		a)
		As $\Theta^{L_V}$ is a closed differential form, by Frobenius theorem and the formula for the exterior derivative, the last condition is equivalent to the constancy of the rank of $\Theta^{L_V}$.
		By the semi-positivity of $\Theta^{L_V}$ and (\ref{eq_dec_max_char}), this is also equivalent to the requirement that the numerical dimension of $L_V$ (see Demailly \cite[Definition (18.13)]{DemaillyAnalMeth}) is equal to $\dim X$.
		\par 
		b) The foliation condition leads to some restrictions on the geometry of $X$, see Beauville \cite{BeauvilleSplit}, and on the growth of certain cohomology groups, see Bouche \cite{Bouche}.
	\end{rem}
	\par 
	Let's now specify previous theorem and consider vector bundles $E_k$, obtained as symmetric powers of a given ample (in the sense of Hartshorne) vector bundle $F$ on a complex manifold $Z$. 
	\par 
	We say that a vector bundle $F$ admits a \textit{projectively flat Hermitian structure} if the associated $PGL(\rk{F}, \comp)$-principal bundle is defined by a representation $\rho : \pi(Z) \to PU(\rk{F})$.
	This is equivalent to the existence of a Hermitian metric $h^F$ on $F$ such that the associated curvature satisfies $R^F = \omega \otimes {\rm{Id}}_F$ for a $(1, 1)$-form $\omega$, cf. Kobayashi \cite[Proposition 1.4.22]{KobaVB}.
	\par 
	We fix a Hermitian metric $h^{\mathcal{O}}$ on the hyperplane line bundle $\mathcal{O}_{\mathbb{P}(F^*)}(1)$ over $\mathbb{P}(F^*)$ with positive curvature (such a metric exists by the ampleness of $F$, see Section \ref{sect_pos_conc}). 
	Consider the natural projection $\pi : \mathbb{P}(F^*) \to Z$ and the $L^2$-metric $h^{S^k F}$ on the symmetric products $S^k F$, $k \in \nat^*$, associated to the volume form induced by the curvature of $h^{\mathcal{O}}$ and the isomorphism 
	\begin{equation}\label{eq_sk_iso}
		S^k F \cong R^0 \pi_* (\mathcal{O}_{\mathbb{P}(F^*)}(k)).
	\end{equation}
	\begin{thm}\label{thm_sk}
	The following assertions are equivalent:
		 \textit{a)} There exists a Hermitian metric $h^{\mathcal{O}}$ on $\mathcal{O}_{\mathbb{P}(F^*)}(1)$  as above
		 such that, as $k \to \infty$, $\overline{\rm{MAVol}}(S^k F, h^{S^k F}) \to 1$.
		\begin{sloppypar}
			 \textit{b)} The vector bundle $F$ admits a projectively flat Hermitian structure. Moreover, if $h^{\mathcal{O}}$ from \textit{a)}  is induced by a Hermtian metric $h^F$ on $F$, then $R^F$ satisfies an identity $R^F = \omega \otimes Id_E$ for some $(1, 1)$-form $\omega$, and in this case we have $\overline{\rm{MAVol}}(S^k F, h^{S^k F}) = 1$ for any $k \in \nat^*$.
		 \end{sloppypar}
	\end{thm}
	\begin{rem}
		A classical calculation shows that if $h^{\mathcal{O}}$ is induced by a Hermitian metric $h^F$  on $F$, then $h^{S^k F}$ corresponds precisely to the Hermitian metric on $S^k F$ induced by $h^F$.
	\end{rem}
	\begin{sloppypar}
		Let's now say few words about the tools used in our proofs.
	In the proof of Theorem \ref{thm_main}, we used some results on Toeplitz operators due to Boutet de Monvel-Guillemin \cite{BoutGuillSpecToepl}, Bordemann–Meinrenken–Schlichenmaier \cite{BordMeinSchli}, \cite{SchliBer}, and the asymptotic result of Ma-Zhang \cite{MaZhangSuperconn}. 
	In the proof of Theorem \ref{thm_charact_max}, we used Deligne pairing due to Deligne \cite{DeligDet} and Zhang \cite{ZhangDelProd}.
	Finally, in the proof of Theorem \ref{thm_sk}, we used the result of Berndtsson, \cite{BernJDG}, on the negativity of direct images.
	\end{sloppypar}
	\par 
	This article is organized as follows. In Section \ref{sect_pos_conc}, we recall notions of ampleness and positivity for vector bundles.
	Sections \ref{sect_main_pf}, \ref{sect_satur}, \ref{sect_symm_pow} are devoted to the proofs of Theorems \ref{thm_main}, \ref{thm_charact_max} and \ref{thm_sk} respectively. Finally, in Appendix \ref{app_toepl_sp}, we give an independent proof of a result about the spectrum of Toeplitz operators, which plays a crucial role in our proof of Theorems \ref{thm_main}.
	
	\par {\bf{Acknowledgements.}} Author would like to express his deepest gratitude for Jean-Pierre Demailly for numerous discussions on the subject of this article.
	He also thanks the colleagues and staff from Institute Fourier, University Grenoble Alps for their hospitality.
	This work is supported by the European Research Council grant ALKAGE number 670846 managed by J.-P. Demailly.
	
		\par {\bf{Notation.}} For two sequences $a_k$, $b_k$, $k \in \nat$, we write $a_k \sim b_k$ if $\lim_{k \to \infty} a_k / b_k = 1$. Denote
		\begin{equation}\label{nk_defn}
			N_k := \dim H^0(X, L^{\otimes k} \otimes G).
		\end{equation}
		For a Hermitian vector bundle $(F, h^F)$, by $R^F := (\nabla^F)^2$ we denote the curvature of the associated Chern connection  $\nabla^F$, and by $\Theta^F := \imun R^F$. By $c_1(F)$ we denote the first Chern class and by $c_1(F, h^F) := {\rm{Tr}} [\Theta^F] / 2 \pi$ the first Chern form.
		Finally, all manifolds in this article are compact.
	\par 
	\section{Ampleness and positivity of vector bundles}\label{sect_pos_conc}
	
	The main goal of this section is to review the notions of ampleness and positivity on vector bundles.
	\par 
	Following Hartshorne, \cite[\S 2]{HartsAmple}, we say that a vector bundle $E$ over a complex manifold $Z$ is \textit{ample} if for every coherent sheaf $\mathscr{F}$, there is an integer $n_0 > 0$, such that for every $n > n_0$, the sheaf $\mathscr{F} \otimes S^n E$ is generated as an $\mathscr{O}_Z$-module by its global sections.
	According to \cite[Proposition 3.2]{HartsAmple}, ampleness of $E$ is equivalent to the ampleness of the line bundle $\mathcal{O}_{\mathbb{P}(E^*)}(1)$ over $\mathbb{P}(E^*)$.
	\par 
	We fix a Hermitian vector bundle $(E, h^E)$, and denote $r := \rk{E}$, $n := \dim Z$. 
	Fix some holomorphic coordinates $(z_1, \ldots, z_n)$ on $X$, an orthonormal frame $e_1, \ldots, e_r$ of $E$, and decompose the curvature as follows
	\begin{equation}\label{eq_curv_dec}
		\imun R^E = \sum_{1 \leq j, k \leq n} \sum_{1 \leq \lambda, \mu \leq r} c_{jk \lambda \mu}  \imun dz_j \wedge d \overline{z}_k \otimes e_{\lambda}^* \otimes e_{\mu}.
	\end{equation}
	We say that $(E, h^E)$ is \textit{Griffiths positive} if the associated quadratic form
	\begin{equation}
		 \widetilde{\Theta}^E(\xi \otimes v) := \langle \imun R^E(\xi, \overline{\xi}) v, v \rangle_{h^E} =  \sum_{1 \leq j, k \leq n} \sum_{1 \leq \lambda, \mu \leq r} c_{jk \lambda \mu}  \xi_j  \overline{\xi}_k v_{\lambda} \overline{v}_{\mu}
	\end{equation}
	takes positive values on non zero tensors $\xi \otimes v \in T^{1, 0}X \otimes E$.
	\textit{Nakano positivity}, see \cite{NakanoPos}, asks that for non zero tensors $\tau = \sum \tau_{j \lambda} \frac{\partial}{\partial z_j} \otimes e_{\lambda}$ in $T^{1, 0}X \otimes E$, the following quadratic form
	\begin{equation}\label{eq_nak_pos}
		\widetilde{\Theta}^E(\tau) := \sum_{1 \leq j, k \leq n} \sum_{1 \leq \lambda, \mu \leq r} c_{jk \lambda \mu} \tau_{j \lambda} \overline{\tau}_{k \mu}  
	\end{equation}
	should take positive values.
	\textit{Dual Nakano positivity stipulates} that the following form
	\begin{equation}\label{eq_dualnak_pos}
		-\widetilde{\Theta}^{F*}(\tau) := \sum_{1 \leq j, k \leq n} \sum_{1 \leq \lambda, \mu \leq r} c_{jk \mu \lambda} \tau_{j \lambda} \overline{\tau}_{k \mu}  
	\end{equation}
	should take positive values on all non zero tensors $\tau = \sum \tau_{j \lambda} \frac{\partial}{\partial z_j} \otimes e_{\lambda}^*$ in $T^{1, 0}X \otimes E^*$.
	\par Griffiths positivity is weaker than Nakano positivity and dual Nakano positivity, cf. Demailly \cite[Proposition 6.6]{DemCompl}. One advantage of the notion of dual Nakano positivity over the notion of Nakano positivity is that it behaves better with respect to taking quotients: a quotient of a dual Nakano positive vector bundle is dual Nakano positive, cf. \cite[Propositon 6.10]{DemCompl}, whereas the analogous statement for Nakano positive vector bundles is false, cf. \cite[Example 6.8, end of \S VII.6]{DemCompl}. 
	Also, in \cite[Basic question 1.7]{DemailluHYMGriff}, Demailly formulated a conjecture stating that the ampleness for a vector bundle is equivalent to the existence of a dual Nakano positive metric. For Nakano positive metrics, the analogical question is known to be false, cf. \cite[(1.8)]{DemailluHYMGriff}.
	\par 
	As we need it later, let's recall an argument of Griffiths showing that the existence of a Griffiths-positive metric $h^F$ on a vector bundle $F$ implies  ampleness.
	We endow the line bundle $\mathcal{O}_{\mathbb{P}(F^*)}(1)$ with the metric $h^{\mathcal{O}}$ induced by $h^F$. 
	We fix a unitary vector $u \in F$ and denote by $u^* \in F^*$ the dual of it.
	We denote by $\omega_{FS, [u^*]}$ the Fubiny-Study form on $[u^*] \in \mathbb{P}(F_{\pi(u)}^*)$ induced by $h^F_{\pi(u)}$.
	Griffiths in \cite{GriffPosVect}, cf. \cite[(15.15)]{DemCompl}, proved that at the point $[u^*] \in \mathbb{P}(F^*)$, the following identity holds
	\begin{equation}\label{eq_griff_formula}
		c_1(\mathcal{O}_{\mathbb{P}(F^*)}(1), h^{\mathcal{O}})_{[u^*]}
		=
		\scal{\frac{\Theta^F}{2 \pi} u}{u}_{h^F} + \omega_{FS, [u^*]}.
	\end{equation}
	\par By the Griffiths-positivity of $(F, h^F)$ and (\ref{eq_griff_formula}), we see  that the metric $h^{\mathcal{O}}$ has positive curvature and thus $\mathcal{O}_{\mathbb{P}(F^*)}(1)$ is ample. By the criteria of Griffiths, it means that $F$ is ample.

	\section{Growth of Monge-Ampère volumes, a proof of Theorem \ref{thm_main}}\label{sect_main_pf}
	The main goal of this section is to give a proof of Theorem \ref{thm_main}. 
	We use notations from Section \ref{sect_intro}. 
	\par 
	Understanding the asymptotics of ${\rm{MAVol}}(E_k, h^{E_k})$ relies heavily on the asymptotics of 
	\begin{equation}\label{eq_rek}
		\Theta^{E_k} \in \ccal^{\infty} \Big( B, \wedge^2 TB \otimes {\rm{End}}\big( H^0(X, L^{\otimes k} \otimes G) \big) \Big).
	\end{equation}
	Toeplitz operators (in the sense of the book of Ma-Marinescu \cite[\S 7]{MaHol}) provide an appropriate framework for studying (\ref{eq_rek}).
	To introduce Toeplitz operators, we denote by 
	\begin{equation}\label{eq_br_proj}
		P_k : L^2(X, L^{\otimes k} \otimes G) \to H^0(X, L^{\otimes k} \otimes G),
	\end{equation}
	the othogonal projection (with respect to (\ref{eq_l2_prod})) from the space of $L^2$-sections to holomorphic sections.
	For a given smooth function $f: X \to \real$, we define 
	\begin{equation}\label{eq_tfk_defn}
		T_{f, k} : H^0(X, L^{\otimes k} \otimes G) \to H^0(X, L^{\otimes k} \otimes G), \qquad T_{f, k}(h) := P_k \circ M_f,
	\end{equation}
	where by $M_f :  H^0(X, L^{\otimes k} \otimes G) \to L^2(X, L^{\otimes k} \otimes G)$, we mean a multiplication by $f$. 
	\begin{sloppypar}
	Following \cite[Definition 7.2.1]{MaHol}, a family of operators $T_k \in \enmr{H^0(X, L^{\otimes k} \otimes G)}$, $k \in \nat^*$, is called a \textit{Toeplitz operator} if for some smooth functions $g_l : Y \to \real$, $l \in \nat$, $T_k$ coincides up to an error of order $O(k^{-\infty})$, as $k \to \infty$, (in the operator norm) with Laurent series $\sum_{l \leq 0} k^l T_{g_l, k}$.
	\end{sloppypar}
	\par 
	A central result in our proof is the theorem of Ma-Zhang \cite{MaZhangSuperconn}, which establishes that $k^{-1} \Theta^{E_k}$, $k \in \nat^*$, is a Toeplitz operator. We only need the following weak version of this theorem, which uses the first term of the Laurent expansion of $k^{-1} \Theta^{E_k}$.
	\begin{thm}[{Ma-Zhang \cite{MaZhangSuperconn} }]\label{thm_ma_zhang}
	There is a constant $C > 0$, such that for any $k \in \nat^*$, we have 
		\begin{equation}
			\Big\| \frac{ \Theta^{E_k}}{2 \pi} - k \cdot T_{\omega_H, k} \Big\| \leq C,
		\end{equation}
		where $\| \cdot \|$ is the operator norm and $\omega_H$ is as in (\ref{eq_omega_h}).
	\end{thm}
	\begin{rem}
		Ma-Zhang in \cite{MaZhangSuperconn} also computed explicitly the next term of the Laurent expansion of $k^{-1} \Theta^{E_k}$.
		Their result even holds under weaker assumption of \textit{relative} ampleness of $L$.
	\end{rem}
	\par 
	The next ingredient in our proof is the spectral theory of Toeplitz operators of Boutet de Monvel-Guillemin \cite[Theorem 13.13]{BoutGuillSpecToepl} (see also Appendix \ref{app_toepl_sp} for an alternative proof of this result), which describes the weak limit, as $k \to \infty$, of the spectral measure of a Toeplitz operator. Recall that $N_k$ was defined in (\ref{nk_defn}). 
	\begin{thm}[{Boutet de Monvel-Guillemin \cite[Theorem 13.13]{BoutGuillSpecToepl} }]\label{thm_bout_guil}
		Let $f : X \to \real$ be  a smooth function and $U \subset \real$ be an open subset satisfying ${\rm{Im}}(f) \subset U$. Fix an arbitrary continuous function $g : U \to \real$.
		Then the following asymptotic holds
		\begin{equation}\label{eq_bout_guil}
			\lim_{k \to \infty} \frac{1}{N_k} \sum_{\lambda \in \spec(T_{f, k})} g(\lambda) 
			=
			 \frac{\int_{X} g(f(x)) \cdot \omega^{\wedge \dim X}(x)}{\int_{X} c_1(L)^{ \dim X}}. 
		\end{equation}
	\end{thm}
	The final ingredient in our proof is the following technical lemma.
	\begin{sloppypar}
	\begin{lem}\label{lem_det}
		Let $(f_{ij}):  X \to \real$, $i, j = 1, \ldots, l$, be a matrix of smooth functions, which is strictly positive definite pointwise. Then we have
		\begin{equation}
			\big| \det \big( (T_{f_{ij}, k}) \big) \big|^{1 / N_k} \sim \big| \det \big( T_{\det(f_{ij}), k} \big) \big|^{1 / N_k},
		\end{equation}		 
		where the first determinant is viewed as a determinant of an operator from $\enmr{H^0(X, L^{\otimes k} \otimes G)^{\oplus l}}$, the second one is a determinant of an operator from $\enmr{H^0(X, L^{\otimes k} \otimes G)}$, and the last one is the pointwise determinant of the matrix $(f_{ij})$.
	\end{lem}
	\end{sloppypar}
	Let's first see how this all is useful in proving Theorem \ref{thm_main}. Then we establish Lemma \ref{lem_det}.
	\begin{proof}[Proof of Theorem \ref{thm_main}.]
		\begin{sloppypar}				
		 We fix a volume form $\nu_B$ over $B$. It defines at every point $x \in B$ a unimodular basis $\partial z_1, \ldots, \partial z_m \in T^{1, 0}B$, $m = \dim B$. By (\ref{eq_defn_det}), we clearly have
		\begin{equation}
			\det{}_{TB \otimes E_k^*}\big(  (\Theta^{E_k})^T / 2 \pi \big)^{1/N_k}
			=
			\det{} \big(  \Theta^{E_k}(\partial z_i, \overline{\partial z}_j)^T / 2 \pi \big)^{1/N_k} 
			\cdot
			m! \cdot
			\nu_B,
		\end{equation}
		where the last determinant is viewed as a determinant of an operator from $\enmr{(H^0(X, L^{\otimes k} \otimes G)^*)^{\oplus m}}$. 
		Clearly, Theorem \ref{thm_ma_zhang} implies that we have
		\begin{equation}
			\det{} \big(  \Theta^{E_k}(\partial z_i, \overline{\partial z}_j)^T / 2 \pi \big)^{1/N_k} 
			\sim
			k^{m}
			\big|
			\det{} \big(  T_{\omega_H(\partial z_i, \overline{\partial z}_j), k} \big) \big|^{1/N_k},
		\end{equation}
		where the second determinant is viewed as a determinant of an operator from $\enmr{H^0(X, L^{\otimes k} \otimes G)^{\oplus m}}$.
		Note that the matrix  $(\omega_H(\partial z_i, \overline{\partial z}_j))$ is strictly positive definite pointwise by the positivity of $(L, h^L)$, so we can apply  Lemma \ref{lem_det}, which gives us
		\begin{equation}
			\big| \det{} \big(  T_{\omega_H(\partial z_i, \overline{\partial z}_j), k} \big) \big|^{1/N_k}
			\sim
			\frac{1}{m!} \cdot
			\big| \det{} \big(  T_{\omega_H^{\wedge m}/\pi^* \nu_B, k} \big) \big|^{1/N_k}.
		\end{equation}
		\end{sloppypar}
		\par  
		By the positivity of $(L, h^L)$, the function $\omega_H^{\wedge m}/\pi^* \nu_B$ is positive.
		Clearly, we have
		\begin{equation}
			\log \big| \det{} \big(  T_{\omega_H(\partial z_i, \overline{\partial z}_j), k} \big) \big|^{1/N_k}
			=
			\frac{1}{N_k} \sum \log(\lambda_{i, k}),
		\end{equation}
		where $\lambda_{i, k}$ enumerates the spectrum of $T_{\omega_H(\partial z_i, \overline{\partial z}_j), k}$ for $k \in \nat^*$.
		So we can apply Theorem \ref{thm_bout_guil} for $g := \log$ to see that, as $k \to \infty$, we have
		\begin{equation}
			\det{} \big(  T_{\omega_H^{\wedge m}/\pi^* \nu_B, k} \big)^{1/N_k}
			\to 
			\exp 
			\Big(
				\frac{\int_X \log \big( \omega_H^{\wedge \dim B} / \pi^* \nu_B \big)
				\omega^{\wedge \dim X}}{\int_X c_1(L)^{\dim X}}
			\Big).
		\end{equation}
		Theorem \ref{thm_main} now follows directly from the above.
	\end{proof}
	\begin{proof}[Proof of Lemma \ref{lem_det}]
	\begin{sloppypar}
		The central idea of the proof is to use a result of Bordemann–Meinrenken–Schlichenmaier \cite{BordMeinSchli}, cf. \cite[Theorem 3.1]{SchliBer}, stating that the algebra, induced by the operators $T_{f, k}$, $f : X \to \real$, coincides asymptotically, as $k \to \infty$, with an algebra of functions on $X$ with standard multiplication. More precisely, in \cite{BordMeinSchli}, \cite{SchliBer}, authors proved that for any smooth functions $f, g : X \to \real$, the product $T_{f, k} \circ T_{g, k}$ is a Toeplitz operator.
		As a consequence of this and explicit evaluation of the first term of Laurent expansion of $T_{f, k} \circ T_{g, k}$, they obtained that there is a constant $C > 0$, such that for any $k \in \nat^*$, the following bound holds
		\begin{equation}\label{eq_commut_tfg}
			\Big\| T_{f, k} \circ T_{g, k} - T_{fg, k} \Big\| \leq \frac{C}{k}.
		\end{equation}
		\par As a sidenote, remark that Ma-Marinescu in \cite{MaMarToepl}, \cite[Theorem 7.4.1]{MaHol} gave an alternative proof (based on the asymptotic expansion of the Bergman kernel) of the above results in a more general setting of symplectic manifolds.
		Then, for Kähler manifolds, Ma-Marinescu in \cite[Theorem 0.3]{MaMarBTKah} calculated explicitly two next terms of the mentionned Laurent expansion, and Ioos in \cite{IoosBTSympl} extended the calculation of the  second term to symplectic manifolds.
		\end{sloppypar}
		\par Now, apply Cholesky algorithm pointwise to the matrix $(f_{ij})$ to represent it in the form $(f_{ij}) = L \cdot L^T$, where $L = (g_{i j})$ is a lower-triangular matrix, and $g_{i j}$ are smooth functions.
		Clearly,
		\begin{equation}\label{eq_def_chol}
			g_{1 1}^{2} \cdot \ldots \cdot g_{l l}^{2} = \det (f_{i j}).
		\end{equation}
		From (\ref{eq_commut_tfg}), we see that for the matrices composed of Toeplitz operators, an asymptotic version of Cholesky decomposition holds, i.e. there is a constant $C > 0$ such that for any $k \in \nat^*$, we have
		\begin{equation}\label{eq_appr_chol}
			\Big\| (T_{f_{ij}, k}) - (T_{g_{i j}, k}) \cdot (T_{g_{j i}, k}) \Big\| \leq \frac{C}{k}.
		\end{equation}
		From (\ref{eq_appr_chol}), we deduce that
		\begin{equation}\label{eq_detf_tog}
			\big| \det (T_{f_{ij}, k}) \big| ^{1/N_k} \sim \big| \det (T_{g_{i j}, k}) \big|^{2/N_k}.
		\end{equation}
		But as the matrix $(g_{i j})$ is lower-triangular, we see that
		\begin{equation}\label{eq_t_eq_prod_l}
			\det (T_{g_{i j}, k}) = \det (T_{g_{1 1}, k}) \cdot \ldots \cdot \det (T_{g_{l l}, k}).
		\end{equation}
		But then again by (\ref{eq_commut_tfg}) and (\ref{eq_def_chol}), we get
		\begin{equation}\label{eq_detg_to_detf}
			\big| \det (T_{g_{1 1}, k}) \cdot \ldots \cdot \det (T_{g_{l l}, k}) \big|^{2/N_k} \sim  \big|\det \big( T_{\det(f_{ij}), k} \big)\big|^{1 / N_k}.
		\end{equation}
		From (\ref{eq_detf_tog}), (\ref{eq_t_eq_prod_l}) and (\ref{eq_detg_to_detf}), we conclude.
	\end{proof}

\section{Saturation in the bound of Demailly, a proof of Theorem \ref{thm_charact_max}}\label{sect_satur}
	The main goal of this section is to prove Theorem \ref{thm_charact_max}. To do so, let's first study the asymptotics of $\rk{E_k}$ and $c_1(E_k)$, as they appear in the denominator of $\overline{\rm{MAVol}}(E_k, h^{E_k})$, (\ref{eq_defn_ma_renorm}). 
	\par By Kodaira vanishing and Riemann-Roch-Grothendieck theorems, we have
	\begin{equation}\label{eq_ch_ek_form}
		\ch(E_k) = \int_{X} \td(TY / B) \ch(L^{\otimes k} \otimes G).
	\end{equation}
	By letting  $k \to \infty$ in (\ref{eq_ch_ek_form}), we see that
	\begin{align}\label{eq_nk_asymp}
		&
		\rk{E_k} \sim k^{\dim X} \frac{\int_X c_1(L)^{\dim X}}{(\dim X)!},
		\\
		&
		c_1(E_k) \sim k^{\dim X + 1} \frac{\pi_* (c_1(L)^{\dim X + 1})}{(\dim X + 1)!}.
	\end{align}
	From (\ref{eq_nk_asymp}), we see directly that
	\begin{equation}\label{eq_asymp_ch_01}
		 \frac{c_1(E_k)}{\rk{E_k}}  \sim k \cdot \frac{\pi_* (c_1(L)^{\dim X + 1})}{(\dim X + 1) \cdot \int_X c_1(L)^{\dim X}}.
	\end{equation}
	\par 
	From Theorem \ref{thm_main}, (\ref{eq_resc_ma_bound}) and (\ref{eq_asymp_ch_01}), we infer that in the notations of Theorem \ref{thm_main}
	\begin{equation}\label{eq_bund_thm_dem}
		\int_B
			\exp 
			\Big(
				\frac{\int_X \log \big( \omega_H^{\wedge \dim B} / \pi^* \nu_B \big)
				\omega^{\wedge \dim X}}{\int_X c_1(L)^{\dim X}}
			\Big)
			d \nu_B
			\leq
			\frac{\int_B (\pi_*  (c_1(L)^{\dim X + 1}))^{\dim B}}{(\dim X + 1)^{\dim B} \cdot (\int_X c_1(L)^{\dim X})^{\dim B}}.
	\end{equation}
	Moreover, $\overline{\rm{MAVol}}(E_k, h^{E_k}) \to 1$, as  $k \to \infty$, if and only if we have an equality in (\ref{eq_bund_thm_dem}). 
	The proof of the following lemma with an independent derivation of (\ref{eq_bund_thm_dem}) is given in the end of this section.
	\begin{lem}\label{lem_eq_case}
		We have an equality in (\ref{eq_bund_thm_dem}) if and only if there exists a smooth function $g : B \to \real$ such that we have
		\begin{equation}\label{eq_eq_case}
			\omega_H = \pi^* \big( g \cdot \pi_* (\omega^{\wedge(\dim X + 1)} ) \big).
		\end{equation}
	\end{lem}
	\par 
	Now let's reformulate the condition (\ref{eq_eq_case}) in more geometrical terms. To do this, we need to review the definition of Deligne pairing due to Deligne \cite{DeligDet} (in relative dimension $1$) and Zhang \cite{ZhangDelProd} (in any relative dimension).
	\par 
	Let $\pi : Y \to B$ be as in Section \ref{sect_intro}, i.e. a proper surjection between two projective manifolds. Denote by $m$ the relative dimension of  $\pi$.
	If $L_0, \ldots , L_m$ are line bundles on $Y$, then the Deligne pairing $\langle L_0, \ldots , L_m \rangle$ is a line bundle on $B$. It is locally generated by symbols $\langle s_0, \ldots , s_m \rangle$, where the $s_j$ are rational sections of the $L_j$ whose divisors, $(s_j)$, have empty intersection. The transition functions are determined as follows. If for some $i$, the intersection $Y = \cap_{j \neq i} (s_j)$ is finite over $B$ and has empty intersection with ${\rm{div}}(f)$ for a given rational function $f$ on $Y$, then 
	\begin{equation}\label{eq_gluing_dw}
		\langle s_0, \ldots, f s_i, \ldots , s_m \rangle
		=
		{\rm{Norm}}_{Y/B}(f)
		\cdot
		\langle s_0, \ldots, s_m \rangle,
	\end{equation}
	where for $b \in B$, $({\rm{Norm}}_{Y/B}(f))(b) := \prod_{y \in \pi^{-1}(b)} f(y)$.
	The fact that (\ref{eq_gluing_dw}) determines a well-defined line bundle follows from Weil reciprocity formula, see Deligne \cite[\S 1.4]{DeligDet}.
	\par 
	It also follows easily from the definition that Deligne pairing is multilinear with respect to the tensor products of its components. 
	In other words, for line bundles $L_0, \ldots , L_m, L'$ over $Y$, for any $0 \leq i \leq m$, there is a canonical isomorphism 
	\begin{equation}\label{eq_multilinear}
		\langle L_0, \ldots,  L_i \otimes L', \ldots,  L_m \rangle \to  \langle L_0, \ldots , L_m  \rangle  \otimes \langle L_0, \ldots,  L_{i - 1},  L', L_{i + 1}, \ldots,  L_m \rangle.
	\end{equation}
	Finally, if $\pi : Y \to B$ is locally an isomorphism ($m = 0$), then  we have a canonical isomorphism
	\begin{equation}\label{eq_m0_del_prod}
		\langle L_0 \rangle \cong \det (R^0 \pi_* L_0).	
	\end{equation}
	\par 
	If $h^{L_j}$ are Hermitian metrics on $L_j$, then one can endow $\langle L_0, \ldots , L_m \rangle$ with the induced canonical Hermitian metric, called Deligne metric.
	It is constructed as follows.
	First, we require Deligne metric to be multilinear in $h^{L_j}$, i.e. compatible with (\ref{eq_multilinear}).
	From this and projectivity of $Y$, it is enough to define Deligne metric only for products of ample line bundles. 
	So we suppose $L_m$ is ample.
	Now, suppose that the divisor ${\rm{div}}(s_m)$ of $s_m$ is smooth (we can always choose $s_m$ in this way by Bertini's theorem).
	Define the Deligne norm $\| \cdot \|$ as follows, cf. \cite[\S 1.2]{ZhangDelProd}, 
	\begin{multline}\label{eq_del_metr_defn}
		\log \|  \langle s_0, \ldots, s_m  \rangle \| 
		=
		\log \|  \langle s_0, \ldots, s_{m-1}  \rangle({\rm{div}}(s_m)) \|
		\\
		+
		\int_X \log \| s_m \| c_1(L_0, h^{L_0}) \wedge \cdots \wedge c_1(L_{m-1}, h^{L_{m-1}}),
	\end{multline}
	where $\langle s_0, \ldots, s_{m-1}  \rangle({\rm{div}}(s_m))$ is an element of the Deligne product $\langle L_0, \ldots , L_{m-1} \rangle$ of the fibration $\pi' : {\rm{div}}(s_m) \to B$, and $\|  \langle s_0, \ldots, s_{m-1}  \rangle({\rm{div}}(s_m)) \|$ is the Deligne norm defined inductively. 
	To fix the basis of induction, we require (\ref{eq_m0_del_prod}) to become an isometry, when the right-hand side is endowed with the metric induced by the restriction of $h^{L_0}$.
	\par 
	Directly from (\ref{eq_del_metr_defn}) and Poincaré-Lelong formula, \cite[(6.6.1)]{DeligDet}, cf. \cite[\S 1.2]{ZhangDelProd}, we see that
	\begin{equation}\label{eq_curv_del}
		c_1 \Big( \langle L_0, \ldots , L_m \rangle, \| \cdot \|^2 \Big) 
		= 
		\pi_* \Big( 
		c_1(L_0, h^{L_0})
		\wedge
		\cdots
		\wedge
		c_1(L_m, h^{L_m})
		\Big).
	\end{equation}
	\begin{proof}[Proof of Theorem \ref{thm_charact_max}]
	\begin{sloppypar}
		Let's first verify that \textit{b)} implies \textit{a)}.
		Indeed (\ref{eq_dec_max_char}) along with the foliation condition imply that 
		\begin{equation}\label{eq_fol_cond_impl}
			p \cdot \omega_H = \pi^* c_1(L_H, h^{L_H}).
		\end{equation}
		By (\ref{eq_fol_cond_impl}), we see that the function under the logarithm in (\ref{eq_main_form}) is constant for $\nu_B := c_1(L_H, h^{L_H})^{\wedge \dim B}$. From this, we see that we get an equality in (\ref{eq_bund_thm_dem}) if 
		\begin{equation}
			p \cdot \pi_*(c_1(L)^{\dim X + 1})  =  (\dim X + 1) \int_X c_1(L)^{\dim X} \cdot c_1(L_H).
		\end{equation}
		But the last identity follows easily from (\ref{eq_fol_cond_impl}).
		\end{sloppypar}
		\par 
		Let's prove the opposite direction. 
		We fix a Hermitian line bundle $(L, h^L)$ satisfying \textit{a)} from Theorem \ref{thm_charact_max}. 
		Then by the discussion after (\ref{eq_bund_thm_dem}), (\ref{eq_bund_thm_dem}) becomes an identity.
		Let's consider the line bundle $L_H := \langle L, \ldots , L \rangle$ over $B$, where $L$ appears $\dim X + 1$ times, and denote by $h^{L_H}$ the Deligne metric on $L_H$ induced by $h^L$.
		By Lemma \ref{lem_eq_case} and (\ref{eq_curv_del}), in the notations of (\ref{eq_eq_case}), we have
		\begin{equation}\label{eq_lh_curv}
			c_1(L_H, h^{L_H}) = g \pi_* (\omega^{\wedge (\dim X + 1)}) \cdot (\dim X + 1) \int_X c_1(L)^{\dim X}.
		\end{equation}
		Clearly, (\ref{eq_eq_case}) and (\ref{eq_lh_curv}) imply (\ref{eq_dec_max_char}) for $p := (\dim X + 1)  \cdot \int_X c_1(L)^{\dim X}$, $L_V := L^{\otimes p} \otimes L_H^*$ and $h^{L_V}$ induced by $h^L$ and $h^{L_H}$.
		The foliation condition follows from (\ref{eq_omega_h}), (\ref{eq_eq_case}) and (\ref{eq_lh_curv}).
	\end{proof}
	
	\begin{proof}[Proof of Lemma \ref{lem_eq_case}]
		The proof of this lemma is inspired by \cite[Proposition 3.2]{DemailluHYMGriff}.
		We denote by $\lambda_1, \ldots, \lambda_{\dim B} : Y \to \real$ the eigenvalues of $\omega_H$ with respect to $\pi^*(\pi_* (\omega^{\wedge(\dim X + 1)}))$.
		Then 		
		\begin{equation}\label{eq_ineq_0}
			\frac{\omega_H^{\wedge \dim B}}{\pi^* (\pi_* (\omega^{\wedge(\dim X + 1)}))^{\wedge \dim B}}
			=
			 \lambda_1 \cdot \ldots \cdot \lambda_{\dim B}.
		\end{equation}
		By the inequality between the arithmetic and geometric means and (\ref{eq_ineq_0}), we have
		\begin{equation}\label{eq_ineq_1}
			\frac{\omega_H^{\wedge \dim B}}{\pi^* (\pi_* (\omega^{\wedge(\dim X + 1)}))^{\wedge \dim B}}
			\leq
			\Big(
				\frac{\lambda_1 + \cdots + \lambda_{\dim B} }{\dim B}
			\Big)^{\dim B}.
		\end{equation}
		By (\ref{eq_ineq_1}) and Jensen's inequality, we obtain
		\begin{multline}\label{eq_ineq_2}
			\exp 
			\Big(
				\int_X \log \big( \omega_H^{\wedge \dim B} / \pi^* (\pi_* (\omega^{\wedge(\dim X + 1)}))^{\wedge \dim B} \big)
				 \frac{\omega^{\wedge \dim X}}{\int_X c_1(L)^{\dim X}}
			\Big)
			\\
			\leq
			\Big[
			\int_X \Big(
				\frac{\lambda_1 + \cdots + \lambda_{\dim B} }{\dim B}
			\Big) \frac{\omega^{\wedge \dim X}}{\int_X c_1(L)^{\dim X}}		
			\Big]^{\dim B}.
		\end{multline}
		We denote by ${\rm{Tr}'}$ the trace of a $(1, 1)$-form on $B$ with respect to $\pi_* (\omega^{\wedge(\dim X + 1)})$. Then
		\begin{equation}\label{eq_ineq_3}
			\pi_* \Big(
				\big( \lambda_1 + \cdots + \lambda_{\dim B} \big)	
				\omega^{\wedge \dim X}
			\Big)
			=
			\frac{			{\rm{Tr}'}\,
			\big[
			\pi_* (\omega^{\wedge (\dim X + 1)})
			\big]}{\dim X + 1}
			=
			\frac{\dim B}{\dim X + 1}.
		\end{equation}
		\par 
		From (\ref{eq_ineq_2}) and (\ref{eq_ineq_3}), by putting $\nu_B := (\pi_* (\omega^{\wedge(\dim X + 1)}))^{\wedge \dim B}$, we get  an independent proof of (\ref{eq_bund_thm_dem}). 
		By investigating the equality case in (\ref{eq_ineq_1}), we see that in case of equality in (\ref{eq_bund_thm_dem}), at each point $x \in Y$, $\omega_H(x)$ has to be equal to $\pi^*(\pi_* (\omega^{\wedge(\dim X + 1)}))(x)$ up to a certain constant.
		By investigating the equality case in (\ref{eq_ineq_2}), we see that in case of equality in (\ref{eq_bund_thm_dem}),  the constant from the previous sentence has to be constant all the way through each fiber. 
	\end{proof}

\section{High symmetric powers of ample vector bundles, a proof of Theorem \ref{thm_sk}}\label{sect_symm_pow}
	The main goal of this section is to give a proof of Theorem \ref{thm_sk}. To do so we will establish a certain characterization of vector bundles admitting projectively flat Hermitian structures.
	\par 
	Recall that a class $\alpha \in H^{1,1}(Z, \comp) \cap H^2(Z, \real)$, in a Kähler manifold $(Z, \omega_0)$ is called \textit{semi-positive} if it contains a semi-positive $(1, 1)$-form.
	Following Demailly \cite[Proposition 4.2]{DemSingHermMetr}, $\alpha$ is called \textit{nef} if for any $\epsilon > 0$, there is a $(1, 1)$-form in $\alpha$, which is strictly bigger than $-\epsilon \omega_0$.
	\par  We conserve the notation for $F$, $\pi : \mathbb{P}(F^*) \to Z$ and $\mathcal{O}_{\mathbb{P}(F^*)}(1)$ from Section \ref{sect_intro}.
	\begin{thm}\label{thm_charac_prfl}
		The class $\Lambda_F := c_1(\mathcal{O}_{\mathbb{P}(F^*)}(1)) - \frac{1}{\rk{F}} \pi^* c_1(F)$ is semi-positive if and only if $F$  admits a projectively flat Hermitian structure.
	\end{thm}
	The proof of Theorem \ref{thm_charac_prfl} is given in the end of this section. Let's put it in in the context of previous works.
	It has been proved by Nakayama in \cite[Theorem A]{NakayamaNef} (and it follows independently from the classification theorem of \textit{numerically flat} vector bundles of  Demailly-Peternell-Schneider \cite[Theorem 1.18]{DPSNef}) that $\Lambda_F$ is \textit{nef} if and only if $F$ admits a filtration by subbundles such that the subsequent quotients of the filtration admit projectively flat Hermitian structures with the same ratio of their first Chern classes to their ranks. Theorem \ref{thm_charac_prfl} says that under a stronger assumption that $\Lambda_F$ is semi-positive, this filtration collapses to a single vector bundle.
	\begin{proof}[Proof of Theorem \ref{thm_sk}]
		Let's first argue that \textit{b)} implies \textit{a)}.
		Indeed, let's assume that $F$ admits a projectively flat Hermitian structure. We fix a Hermitian metric $h^F$ satisfying $\Theta^F = \omega \otimes {\rm{Id}}_F$ for some $(1, 1)$-form $\omega$.
		Then by (\ref{eq_griff_formula}), we see that for the associated metric $h^{\mathcal{O}}$ on the line bundle $\mathcal{O}_{\mathbb{P}(F^*)}(1)$, the horizontal component of the first Chern form satisfies
		\begin{equation}\label{eq_c_1_0_griff}
			c_1(\mathcal{O}_{\mathbb{P}(F^*)}(1), h^{\mathcal{O}})_H = \frac{\pi^* \omega}{2 \pi}.
		\end{equation}
		 We conclude by (\ref{eq_c_1_0_griff}) in the same way as we did in the first part of the proof of Theorem \ref{thm_charact_max}.
		\par 
		Let's prove the opposite direction. 
		First of all, from (\ref{eq_dec_max_char}), we see that the class $c_1(\mathcal{O}_{\mathbb{P}(F^*)}(1)) - \frac{1}{p}\pi^* c_1(L_H)$ is semi-positive.
		By the proof of Theorem \ref{thm_charact_max}, we see that $L_H$ can be chosen so that 
		\begin{equation}\label{eq_15_1}
			c_1(L_H) = \pi_*(c_1(\mathcal{O}_{\mathbb{P}(F^*)}(1))^{\rk{F}}),
		\end{equation}
		and $p := \rk{F}$.
		From Grothendieck approach to Chern classes, we see that
		\begin{equation}\label{eq_groth_ch}
			 \pi_*(c_1(\mathcal{O}_{\mathbb{P}(F^*)}(1))^{\rk{F}})
			 =
			 c_1(F).
		\end{equation}
		From the above, we infer the semi-positivity of $\Lambda_F$. As a consequence, by Theorem \ref{thm_charac_prfl}, $F$ admits a projectively flat Hermitian structure.
		\par 
		Now let's assume that $h^{\mathcal{O}}$ is induced by a Hermitian metric $h^F$ on $F$.
		The foliation condition from Theorem \ref{thm_charac_prfl} implies in the notation of (\ref{eq_c_1_0_griff}) that 
		\begin{equation}\label{eq_fol_cond_ref}
			c_1(\mathcal{O}_{\mathbb{P}(F^*)}(1), h^{\mathcal{O}})_{H} = \frac{1}{p} \pi^* c_1(L_H, h^{L_H}).
		\end{equation}
		But the formula of Griffiths, (\ref{eq_griff_formula}), says that for a given unitary vector $u \in F$, we have 
		\begin{equation}\label{eq_fol_cond_ref211}
			c_1(\mathcal{O}_{\mathbb{P}(F^*)}(1), h^{\mathcal{O}})_{H, [u^*]} = \scal{\frac{\Theta^F}{2 \pi}  u}{u}_{h^F}.
		\end{equation}
		From (\ref{eq_fol_cond_ref}) and (\ref{eq_fol_cond_ref211}), we have
		$
			\frac{\Theta^F}{2 \pi} = \frac{1}{p} c_1(L_H, h^{L_H}) \otimes {\rm{Id}}_F,
		$
		which finishes the proof.
	\end{proof}
	\begin{proof}[Proof of Theorem \ref{thm_charac_prfl}]
		The main idea of the proof is to start from a semi-positive form $\alpha$ in the class $\Lambda_F$ and by using the isomorphism (\ref{eq_sk_iso}) for $k = 1$, construct a Hermitian metric $h^F$ on $F$ through the $L^2$-product, so that $(F, h^F)$ is projectively flat.
		\par
		More precisely, we fix an arbitrary metric $h^{F}_{0}$ on $F$ so that the induced metric on $\det F$ has positive curvature and endow $\mathcal{O}_{\mathbb{P}(F^*)}(1)$ with the metric $h^{\mathcal{O}}$, satisfying
		\begin{equation}\label{eq_51_thm_1}
			c_1(\mathcal{O}_{\mathbb{P}(F^*)}(1), h^{\mathcal{O}}) 
			=
			\alpha 
			+
			\frac{\pi^* c_1(F, h^{F}_{0})}{\rk{F}}.
		\end{equation}
		Note that it is always possible by the $\partial \dbar$-lemma.
		\par 
		We specify (\ref{eq_sk_iso}) for $k = 1$ to get an isomorphism
		\begin{equation}\label{eq_f_as_direct_im}
			F \cong R^0 \pi_* (\mathcal{O}_{\mathbb{P}(F^*)}(1)).
		\end{equation}
		We denote by $h^F$ the $L^2$-metric on $F$, induced by  $h^{\mathcal{O}}$ and the fiberwise volume form obtained by the restriction of $\alpha^{\rk{F}-1}$. Since $\alpha$ is in the class $\Lambda_F$, we have $\pi_* (\alpha^{\rk{F}-1}) = 1$, and the volume form from the previous sentence is non-degenerate at some point of each fiber. Hence the metric $h^F$ is non-degenerate everywhere.
		We claim that $(F, h^F)$ is projectively flat.
		\begin{sloppypar} 
		To see this, we use the result of Berndtsson, \cite[Theorem 3.1]{BernJDG}, on the negativity of direct image bundles of locally trivial fibrations, which states that the curvature $\Theta^F$ of $h^F$ for any $u \in F$ satisfies
		\begin{equation}\label{eq_51_thm_2}
			\scal{\frac{\Theta^F}{2 \pi}  u}{u}_{h^F} \leq \pi_* \Big( h^{\mathcal{O}}(\tilde{u}, \tilde{u}) \cdot c_1(\mathcal{O}_{\mathbb{P}(F^*)}(1), h^{\mathcal{O}}) \wedge \alpha^{\wedge(\rk{F}-1)} \Big),
		\end{equation}
		where $\tilde{u}$ is an element of $R^0 \pi_* (\mathcal{O}_{\mathbb{P}(F^*)}(1))$, associated to $u$ by the isomorphism (\ref{eq_f_as_direct_im}).
		Remark that in \cite[Theorem 3.1]{BernJDG} it is required that the relative volume form is induced by $c_1(\mathcal{O}_{\mathbb{P}(F^*)}(1), h^{\mathcal{O}})$, but by following the argument in the proof of \cite[Theorem 3.1]{BernJDG} line by line, we see that it works in our situation as well, as the only thing which is used by Berndtsson is the closedness of $\alpha$.
		\end{sloppypar}		
		\par 
		From (\ref{eq_51_thm_1}), (\ref{eq_51_thm_2}) and the definition of the $L^2$-scalar product, we have
		\begin{equation}\label{eq_bern_bd_last}
			\Big\langle \Big( \frac{\Theta^F}{2 \pi} - \frac{\pi^* c_1(F, h^{F}_{0}) \otimes {\rm{Id}}_F}{\rk{F}} \Big)  u, u \Big\rangle_{h^F} 
			\leq 
			\pi_* \big( h^{\mathcal{O}}(\tilde{u}, \tilde{u}) \cdot \alpha^{\wedge \rk{F}} \big)
			.
		\end{equation}
		Now, from (\ref{eq_groth_ch}) and the fact that $\alpha$ is in the class $\Lambda_F$, the De Rham cohomology of the form $\pi_* ( \alpha^{\wedge \rk{F}})$ is zero. 
		By the semi-positivity of $\alpha$, the form $\pi_* ( \alpha^{\wedge \rk{F}})$ is zero itself, and, moreover, for any function $g : \mathbb{P}(F^*) \to \real$,  we have $\pi_*( g \cdot \alpha^{\wedge \rk{F}}) = 0$. This with (\ref{eq_bern_bd_last}) implies 
		\begin{equation}\label{eq_bern_bd_last2}
			\Big\langle \Big( \frac{\Theta^F}{2 \pi} - \frac{\pi^* c_1(F, h^{F}_{0}) \otimes {\rm{Id}}_F}{\rk{F}} \Big)  u, u \Big\rangle_{h^F} 
			\leq 
			0
			.
		\end{equation}
		\par 
		By taking the trace of the left-hand side of (\ref{eq_bern_bd_last2}), we get $c_1(F, h^F)  - c_1(F, h^{F}_{0}) \leq 0$. But then we have $c_1(F, h^F)  = c_1(F, h^{F}_{0})$ by the compactness of $Z$. This implies that the left-hand side of (\ref{eq_bern_bd_last2}) is actually equal to zero. This finishes the proof, as it gives us $\frac{\Theta^F}{2 \pi} = \frac{\pi^* c_1(F, h^{F}_{0})}{\rk{F}} \otimes {\rm{Id}}_F$.
	\end{proof}

\appendix 
\section{Spectral measure of Toeplitz operators, a proof of Theorem \ref{thm_bout_guil}}\label{app_toepl_sp}
The main goal of this section is to give a proof of Theorem \ref{thm_bout_guil}.
\par 
The original proof of Boutet de Monvel-Guillemin from \cite{BoutGuillSpecToepl} passes through the proof of so-called trace formula and it relies heavily on  the theory of Fourier integral operators.
	Our proof is based only on the asymptotic expansion of the Bergman kernel (which can be established independently of \cite{BoutGuillSpecToepl}, see Dai-Liu-Ma \cite{DaiLiuMa}).
	Note, however, that this asymptotic  expansion was proved after \cite{BoutGuillSpecToepl} (and some earlier proofs of it were influenced by \cite{BoutGuillSpecToepl}, see Zelditch \cite{ZeldBerg}).
	\par 
	Let's first review the Bergman kernel expansion.
	Let $X$ be a complex manifold of dimension $m$ and let $L$ be an ample line bundle over $X$ endowed with a Hermitian metric $h^L$ of positive curvature. 
	As before, we denote $\omega := \frac{\Theta^L}{2\pi}$.
	Let $(G, h^G)$ be a Hermitian vector bundle on $X$.
	Let
	\begin{equation}
	P_k(x, x') \in (L^{\otimes k} \otimes G)_x \otimes (L^{\otimes k} \otimes G)^*_{x'}, \qquad x, x' \in X,
	\end{equation}
	be the Schwartz kernel (also called the Bergman kernel) of (\ref{eq_br_proj}) with respect to the volume form $\frac{\omega^{\wedge m}}{m!}$.
	For simplicity, assume that the volume form in the $L^2$-product (\ref{eq_l2_prod}) is also induced by $\frac{\omega^{\wedge m}}{m!}$.
	\begin{prop}[{Dai-Liu-Ma \cite[Proposition 4.1]{DaiLiuMa}, cf. Ma-Marinescu \cite[Proposition 4.1.5]{MaHol}}]\label{prop_off_diag_rough}
		For any $\epsilon > 0$, $p, l \in \nat^*$, there is $C > 0$ such that for any $k \in \nat$, $x, x' \in X$, $\dist(x, x') > \epsilon$, we have
		\begin{equation}
			|P_k(x, x')|_{\ccal^{l}} \leq C k^{-p}.
		\end{equation}
	\end{prop}

	Another theorem from \cite{DaiLiuMa} gives a very precise estimate for the Bergman kernel near the diagonal.
	To state it, for $x \in X$, we implicitly identity $T_x X$ and a neighborhood of $x \in X$ by means of the exponential mapping associated to the Kähler form $\omega$.
	We denote by $\kappa$ a function, defined in the neighborhood of $0$ in $T_x X$, which describes the defect between the volume form on $T_x X$ induced by $\omega_x$ and the pull-back using the exponential map of the volume form induced by $\omega$ on $X$.
	As the derivative of the exponential map at $0$ is identity, $\kappa(0) = 1$.
	\begin{thm}[{Dai-Liu-Ma \cite[Theorem 4.1.18']{DaiLiuMa}, cf. Ma-Marinescu \cite[Theorem 4.2.1, (4.1.84), (4.1.92)]{MaHol}}]\label{thm_dai_liu_ma}
		For certain $c, C, M > 0$ and any $x \in X$, $Z, Z' \in T_x X$, the following bound holds
		\begin{multline}\label{eq_dai_liu_ma}
			\Big| 
				P_k(Z, Z') - k^m \exp \big( - \pi k | Z - Z' |^2/2 + \pi k \scal{Z}{Z'} \big) \kappa(Z)^{-1/2} \kappa(Z')^{-1/2}
			\Big|
			\\
			\leq
			C k^{m-1}
			(1 + \sqrt{k}|Z| + \sqrt{k}|Z'|)^M 
			\exp \big(
				-c \sqrt{k} |Z - Z'|
			\big)
			+ 
			C k^{m-1}.
		\end{multline}
	\end{thm}
	\begin{proof}[Proof of Theorem \ref{thm_bout_guil}]
		By Stone-Weierstrass theorem, it is enough to establish (\ref{eq_bout_guil}) only for $g(x) = x^l$, $l \in \nat$. 
		Clearly, (\ref{eq_bout_guil}) holds for $g(x) = 1$.
		To see that it holds for $g(x) = x$, it is enough to interpret the sum on the left-hand side of (\ref{eq_bout_guil}) through the trace of the operator $T_{f, k}$ and to use (\ref{eq_dai_liu_ma}) for $Z, Z' = 0$. We leave this simple verification for the reader.
		\par 
		Now, for $g(x) = x^l$, $l \geq 2$, it is necessary to use (\ref{eq_dai_liu_ma}) for nonzero $Z, Z'$.
		To simplify the exposition, we will only verify (\ref{eq_bout_guil}) for $g(x) = x^2$, as higher powers are treated in a completely analogous way, but the notation becomes more involved.
		First of all, we have
		\begin{equation}\label{eq_spec_0}
			\sum_{\lambda \in \spec(T_{f, k})} \lambda^2 
			=
			 \tr{T_{f, k}^{2}}.
		\end{equation}
		By definition of $T_{f, k}$ in (\ref{eq_tfk_defn}), we may rewrite
		\begin{equation}\label{eq_spec_2}
			\tr{T_{f, k}^{2}} =
			\int_{x \in X} \int_{y \in X}
			f(x)f(y)
			P_k(x, y) \cdot
			P_k(y, x)
			\frac{\omega^{\wedge m}(x)}{m!}
			\frac{\omega^{\wedge m}(y)}{m!}.
		\end{equation}
		Now, by Proposition \ref{prop_off_diag_rough}, the asymptotic expansion of the right-hand side of (\ref{eq_spec_2}) localizes near the diagonal $\{ x = y \}$.
		Thus, for a given $\epsilon > 0$, we may discard the integration over the pairs $x, y \in X$, for which $\dist(x, y) > \epsilon$.
		But for the pairs $x, y \in X$, satisfying $\dist(x, y) < \epsilon$, we may apply Theorem \ref{thm_dai_liu_ma} to study (\ref{eq_spec_2}).
		By discarding the lower order terms (as those introduced by Taylor expansions of $\kappa - 1$ or $f(x) - f(y)$), we get
		\begin{multline}\label{eq_spec_22}
			\int_{x \in X} \int_{y \in X}
			f(x)f(y)
			P_k(x, y) \cdot
			P_k(y, x)
			\frac{\omega^{\wedge m}(x)}{m!}
			\frac{\omega^{\wedge m}(y)}{m!}
			\sim
			k^{2m} 
			\cdot
			\int_{x \in X}
			f(x)^2
			\frac{\omega^{\wedge m}(x)}{m!}
			\cdot
			\\
			\cdot
			\int_{Z \in \real^{2m}} 
			\exp(- \pi k |Z|^2)
			 dZ_1 \wedge d Z_2 \wedge \ldots \wedge  d Z_{2m}.
		\end{multline}
		By the formula $\int_{(x, y) \in \real^2} \exp(- \pi (x^2 + y^2)) dx \wedge d y = 1$, (\ref{eq_spec_2}) and (\ref{eq_spec_22}), we see that
		\begin{equation}\label{eq_spec_3}
			\tr{T_{f, k}^{2}}
			\sim
			k^m
			\cdot
			\int_{x \in X} 
			f(x)^2
			\frac{\omega^{\wedge m}(x)}{m!}.
		\end{equation}
		By (\ref{eq_nk_asymp}), (\ref{eq_spec_0}) and (\ref{eq_spec_3}), we conclude.
	\end{proof}

\bibliography{bibliography}

		\bibliographystyle{abbrv}

\Addresses

\end{document}